\documentclass{article}
\usepackage[english]{babel}
\usepackage{amsmath}
\usepackage{amssymb}
\usepackage{amsfonts}
\usepackage{amsthm}
\usepackage{graphicx}

\usepackage{url}

\def\thtext#1{
  \catcode`@=11
  \gdef\@thmcountersep{. #1}
  \catcode`@=12
}

\def\threst{
  \catcode`@=11
  \gdef\@thmcountersep{.}
  \catcode`@=12
}

\theoremstyle{plain}
\newtheorem{thm}{Theorem}[section]
\newtheorem{prop}[thm]{Proposition}
\newtheorem{cor}[thm]{Corollary}
\newtheorem{lem}[thm]{Lemma}

\theoremstyle{definition}

 \catcode`@=11
 \def\.{.\spacefactor\@m}
 \catcode`@=12

\def\R{\mathbb R}

\def\a{\alpha}

\def\D{\Delta}
\def\g{\gamma}

\def\l{\lambda}

\def\0{\emptyset}
\def\:{\colon}
\def\<{\langle}
\def\>{\rangle}

\def\rom#1{\emph{#1}}
\def\({\rom(}
\def\){\rom)}

\def\ss{\subset}

\def\bcAD{\overline{\mathstrut\cAD}}

\def\diam{\operatorname{diam}}

\def\Ext{\operatorname{Ext}}

\def\cAD{\mathcal{AD}}

\def\cD{\mathcal{D}}

\def\cM{\mathcal{M}}

\begin{document}
\title{The Gromov--Hausdorff Distance between Simplexes and Two-Distance Spaces}
\author{A.\,O.~Ivanov, and A.\,A.~Tuzhilin}
\date{}
\maketitle

\begin{abstract}
In the present paper we calculate the Gromov--Hausdorff distance between an arbitrary simplex  (a metric space all whose non-zero distances are the same) and a finite metric space whose non-zero distances take two distinct values (so-called $2$-distance spaces). As a corollary, a complete solution to generalized Borsuk problem for the $2$-distance spaces is obtained. In addition, we derive formulas for the clique covering number and for the chromatic number of an arbitrary graph $G$ in terms of the Gromov--Hausdorff distance between a simplex and an appropriate $2$-distance space constructed by the graph $G$.
\end{abstract}

\section*{Introduction}
\markright{Introduction}
Finite metric spaces appear in different branches of mathematics, play important role in applications, give many interesting problems, and they are a subject of active research for many specialists, see numerous examples in~\cite{DezaDeza}. An important class of these spaces is formed by the ones where the distance between different points takes two values (to be short, those spaces are referred as \emph{$2$-distance spaces\/}). Finite subsets of Euclidean spaces having this property are of special interest. The fact is that they are closely related to such popular branches as spherical codes and spherical design, see for example~\cite{Musin}, and Borsuk problem too. Let us concentrate on the latter one in more details. 

Recall that as early as 1933 Borsuk formulated a conjecture that each bounded non single-point subset of Euclidean space $\R^n$ can be partitioned into $k\le n+1$ subsets whose diameters are less than the dimeter of the initial set, and proved it for  $n=2$, see~\cite{Borsuk1},~\cite{Borsuk2}. Perkal~\cite{Perkal} extended the result to the case $n=3$. Also some essential progress was made in 40--50th. Let us mention Hadwiger, who proved the conjecture for convex subsets, see~\cite{Hadw1},~\cite{Hadw2}. However, in 1993 Kahn and Kalai ~\cite{KahnKalai} constructed an unexpected counterexample in dimension $n=1325$, and also proved that the conjecture is not valid for all $n>2014$.  This estimate was consistently improved by Raigorodskii, $n\ge561$, Hinrichs and Richter, $n\ge298$, Bondarenko, $n\ge 65$, and Jenrich, $n\ge 64$, see details in a review~\cite{Raig}. The results of Bondarenko~\cite{Bond} and Jenrich~\cite{Jenr} are based on the $2$-distance subsets of the unit sphere.

Recently the authors discovered, see~\cite{IvaTuzBorsuk}, that Generalized Borsuk Problem (the same question for an arbitrary bounded metric space $X$, which is not necessary a subset of $\R^n$) can be solved by means of the Gromov--Hausdorff distance between $X$ and so-called simplexes, i.e., metric spaces all whose non-zero distances are the same (see Theorem~\ref{thm:Borsuk} below).

Recall that Edwards~\cite{Edwards} and Gromov~\cite{Gromov} defined independently a distance function between metric spaces generalising the construction of Hausdorff and using isometrical embeddings into all possible ambient spaces, see definitions below. A detailed introduction to geometry of the Gromov--Hausdorff distance can be found in~\cite[Ch.~7]{BurBurIva} or in~\cite{ITlectHGH}.

Calculation of the Gromov--Hausdorff distance between two given spaces is a rather non-trivial problem. In paper~\cite{IvaTuzSimpDist} the distances between finite simplexes and compact metric spaces were calculated for some particular cases. Later on, these results were generalized to the case of arbitrary bounded metric spaces, see~\cite{GrigIvaTuz}. Some additional characteristics of the bounded metric spaces were defined, and in this terms either exact formulas for the Gromov--Hausdorff distance to an arbitrary simplex, or an exact upper and low estimates for these distances were obtained. Just those formulas permitted to discover a relations between the distances to simplexes and Borsuk problem, see~\cite{IvaTuzBorsuk}.  After that, using a geometrical interpretation, formulas from~\cite{GrigIvaTuz} have been rewritten in a more convenient way that gives an opportunity to calculate the distances between simplexes and an arbitrary finite ultrametric space~\cite{IvaTuzUltra}.

In the present paper, formulas from~\cite{IvaTuzUltra} are used to calculate the Gromov--Hausdorff distances from simplexes to finite $2$-distance spaces (Theorem~\ref{thm:main}). That permits to obtain a complete solution to generalized Borsuk problem for such spaces. The answer  (Corollary~\ref{cor:Borsuk}) is given in terms of the clique covering number of the graph $G$ with the vertex set $X$ and edge set consisting of all the pairs of points from $X$ with the smaller distance. Notice that the idea to verify Borsuk conjecture for $2$-distance spaces was suggested by David Larman in 70th.

Besides, the same technique gives an opportunity to calculate the clique covering number and the chromatic number of an arbitrary simple graph in terms of the Gromov--Hausdorff distance from simplexes to a finite $2$-distance metric space constructed by this graph, see Corollaries~\ref{cor:chrom} and~\ref{cor:clique}.

The work is partly supported by President RF Program supporting leading scientific schools of Russia (Project NSh--6399.2018.1, Agreement~075--02--2018--867), by RFBR, Project~19-01-00775-a, and also by MGU scientific schools support program.

\section{Preliminaries}
\markright{\thesection.~Preliminaries}
Let $X$ be an arbitrary set. By $\#X$ we denote the \emph{cardinality\/} of the set $X$.

Let $X$ be an arbitrary metric space. The distance between any its points $x$ and $y$ we denote by $|xy|$. If $A,B\ss X$ are non-empty subsets of $X$, then put $|AB|=\inf\bigl\{|ab|:a\in A,\,b\in B\bigr\}$. For $A=\{a\}$, we write $|aB|=|Ba|$ instead of $|\{a\}B|=|B\{a\}|$.

For each  point $x\in X$ and a number $r>0$, by $U_r(x)$ we denote the open ball with center $x$ and radius $r$; for any non-empty $A\ss X$ and a number $r>0$ put $U_r(A)=\cup_{a\in A}U_r(a)$.

\subsection{Hausdorff and Gromov--Hausdorff Distances}
For non-empty $A,\,B\ss X$ put
\begin{multline*}
d_H(A,B)=\inf\bigl\{r>0:A\ss U_r(B),\ \text{and}\ B\ss U_r(A)\bigr\}\\ =\max\{\sup_{a\in A}|aB|,\ \sup_{b\in B}|Ab|\}.
\end{multline*} 
This value is called the \emph{Hausdorff distance between $A$ and $B$}. It is well-known, see~\cite{BurBurIva}, \cite{ITlectHGH}, that the Hausdorff distance is a metric on the set of all non-empty bounded closed subsets of $X$.

Let $X$ and $Y$ be metric spaces. A triple $(X',Y',Z)$ consisting of a metric space $Z$ together with its subsets $X'$ and $Y'$ isometric to $X$ and $Y$, respectively, is called a \emph{realization of the pair $(X,Y)$}. The \emph{Gromov--Hausdorff distance $d_{GH}(X,Y)$ between $X$ and $Y$} is the infimum of real numbers $r$ such that there exists a realization  $(X',Y',Z)$ of the pair $(X,Y)$ with $d_H(X',Y')\le r$. It is well-known~\cite{BurBurIva}, \cite{ITlectHGH}, that $d_{GH}$ is a metric on the set $\cM$ of all compact metric spaces considered up to an isometry.

For an arbitrary metric space $X$, by $\diam X$ we denote its \emph{diameter\/} defined in the standard way:
$$
\diam X=\sup\bigl\{|xy|:x,y\in X\bigr\}.
$$
Notice that the space $X$ is bounded if and only if $\diam X<\infty$.

A metric space $X$ is called \emph{simplex\/} if all its non-zero distances are the same. A simplex whose non-zero distances equal $\l>0$ is denoted by $\l\D$.

\begin{prop}[\cite{BurBurIva}, \cite{ITlectHGH}]\label{prop:GH_simple}
Let $X$ be an arbitrary metric space, and $\D$ a single-point space. Then $d_{GH}(\l\D,X)=\frac12\diam X$ for any  $\l>0$.
\end{prop}

\begin{thm}[\cite{GrigIvaTuz}]\label{thm:dist-n-simplex-bigger-dim}
Let $X$ be an arbitrary bounded metric space, and $\#X<\#\l\D$, then
$$
2d_{GH}(\l\D,X)=\max\{\diam X-\l,\,\l\}.
$$
\end{thm}

Let $X$ be an arbitrary set and $m$ a cardinal number that does not exceed $\#X$. By  $\cD_m(X)$ we denote the family of all possible partitions of the set $X$ into $m$ non-empty subsets.

Now let $X$ be a metric space. Then for each $D=\{X_i\}_{i\in I}\in\cD_m(X)$ put
$$
\diam D=\sup_{i\in I}\diam X_i.
$$
Further, for each $D=\{X_i\}_{i\in I}\in\cD_m(X)$ put
$$
\a(D)=\inf\bigl\{|X_iX_j|:i\ne j\bigr\}.
$$

In the plane with the standard coordinates $(\a,d)$, consider the set
$$
\cAD_m(X)=\Bigl\{\bigl(\a(D),\diam D\bigr):D\in\cD_m(X)\Bigr\}
$$
and let $\bcAD_m(X)$ stand for the closure of this subset of the plane. A point $(\a,d)\in\bcAD_m(X)$ is said to be  \emph{extreme\/} if there is no other point $(\a',d')\in\bcAD_m(X)$, $(\a',d')\ne (\a,d)$, such that $\a'\ge\a$ and $d'\le d$. By $\Ext_m(X)$ we denote the set of all extreme points from $\bcAD_m(X)$. In addition, put $h_{\a,d}(\l)=\max\{d,\,\l-\a\}$.

\begin{thm}[\cite{IvaTuzUltra}]\label{thm:extr}
Let $X$ be an arbitrary bounded metric space, and $m=\#\l\D\le\#X$. Then the set $\Ext_m(X)$ is non-empty, and
$$
2d_{GH}(\l\D,X)=\max\bigl\{\diam X-\l,\,\inf_{(\a,d)\in\Ext_m(X)}h_{\a,d}(\l)\bigr\}.
$$
\end{thm}

\section{The Gromov--Hausdorff Distance between Simplexes and $2$-distance Spaces}
\markright{\thesection.~The distance between simplexes and $2$-distance spaces}
Now, let us show how the formulas for the Gromov--Hausdorff distances between simplexes and $2$-distance spaces can be derived from Theorems~\ref{thm:dist-n-simplex-bigger-dim} and~\ref{thm:extr}.

\textbf{In this Section, $X$ always stands for a finite $2$-distance space whose non-zero distances are $a$ and $b$, $a<b$; in addition put $n=\#X$ and $m=\#\l\D$}. It is clear that $\diam X=b$.

Theorem~\ref{thm:dist-n-simplex-bigger-dim} implies the following result.

\begin{cor}
Let $m>n$. Then
$$
2d_{GH}(\l\D,X)=\max\{b-\l,\,\l\}.
$$
\end{cor}

If $m=n$, then $\cD_m(X)$ consists of a single element $D$, namely, $D$ is the partition of $X$ into single-element subsets. Thus, in this case, $\diam D=0$ and $\a(D)=a$. Theorem~\ref{thm:extr} immediately implies the following result.

\begin{cor}
Let $m=n$. Then
$$
2d_{GH}(\l\D,X)=\max\{b-\l,\,\l-a\}.
$$
\end{cor}

Another trivial case is $m=1$. Proposition~\ref{prop:GH_simple} implies the next formula.

\begin{cor}
Let $m=1$. Then
$$
2d_{GH}(\l\D,X)=b.
$$
\end{cor}

It remains to consider the case $1<m<n$. Here for each $D\in\cD_m(X)$ the characteristics $\diam D$ and $\a(D)$ take one of the values $a$ and $b$, therefore, in this case the sets $\Ext_m(X)\ss\bcAD_m(X)=\cAD_m(X)$ are contained in the corresponding four-element subset of the plane. The next Lemma follows directly from the definition of an extreme point.

\begin{lem}\label{lem:extremeCases}
For $1<m<n$ the following situations are possible\/\rom:
\begin{enumerate}
\item If $(b,a)\in\cAD_m(X)$, then $\Ext_m(X)=\bigl\{(b,a)\bigr\}$\rom;
\item If $(b,a)\not\in\cAD_m(X)$, then there are the  next possibilities\/\rom:
\begin{enumerate}
\item If $\#\cAD_m(X)=3$, then $\Ext_m(X)=\cAD_m(X)=\bigl\{(a,a),\,(b,b)\bigr\}$\rom;
\item If $\cAD_m(X)=\bigl\{(a,a),\,(b,b)\bigr\}$, then 
$$
\Ext_m(X)=\cAD_m(X)=\bigl\{(a,a),\,(b,b)\bigr\};
$$
\item If $\cAD_m(X)=\bigl\{(a,a),\,(a,b)\bigr\}$, then $\Ext_m(X)=\bigl\{(a,a)\bigr\}$\rom;
\item If $\cAD_m(X)=\bigl\{(a,b),\,(b,b)\bigr\}$, then $\Ext_m(X)=\bigl\{(b,b)\bigr\}$\rom;
\item If $\#\cAD_m(X)=1$, then $\Ext_m(X)=\cAD_m(X)$.
\end{enumerate}
\end{enumerate}
In particular,    
\begin{enumerate}
\item If $\diam D=b$ and $(b,b)\in\cAD_m(X)$ for any $D\in\cD_m(X)$, then $\Ext_m(X)=\bigl\{(b,b)\bigr\}$\rom;
\item If $\a(D)=a$ and $(a,a)\in\cAD_m(X)$ for any $D\in\cD_m(X)$, then $\Ext_m(X)=\bigl\{(a,a)\bigr\}$.
\end{enumerate}
\end{lem}

Recall that a subgraph of an arbitrary simple graph $H$ is called \emph{clique}, if any its two vertices are connected by an edge, i.e., a clique is a subgraph which is a complete graph itself. Notice that each single-vertex subgraph is also a clique. For convenience, the vertex set of a clique is also referred as a \emph{clique}.

On the set of all cliques an ordering with respect to inclusion is naturally defined, and hence, due to the above remarks, a family of maximal cliques is uniquely defined; this family forms a   \emph{covering of the graph} $H$ in the following sense: the union of all vertex sets of all maximal cliques coincide with the vertex set $V(H)$ of the graph $H$.

If one does not restrict himself by maximal cliques, then, generally speaking, one can find other families of cliques covering the graph $H$. One of the classical problems of Graph Theory is to calculate the minimal possible number of cliques covering a graph $H$. This number is referred as the \emph{clique covering number\/} and is often denoted by $\theta(H)$. It is easy to see that the value $\theta(H)$ is also equal to the least number of cliques whose vertex sets partition  $V(H)$.

Another popular problem is to find the least possible number of colors that is necessary to color the vertices of a simple graph $H$ in such a way that adjacent vertices have different colors. This number is denoted by $\g(H)$ and is referred as the \emph{chromatic number of the graph  $H$}.

For a simple graph $H$, by $H'$ we denote its \emph{dual\/} graph, i.e., the graph having the same vertex set and the complementary set of edges (two vertices are adjacent in $H'$ if and only if they are not adjacent in $H$).

The next fact is well-known.

\begin{prop}
For any simple graph $H$, the equality $\theta(H)=\g(H')$ holds. 
\end{prop}

By $k(H)$ we denote the number of connected components of the graph $H$.

\begin{lem}\label{lem:ComponEdClique}
For any simple graph $H$, the inequality $k(H)\le\theta(H)$ is valid, and the equality holds if and only if the graph $H$ coincides with the disjoint union of its maximal cliques\/\rom: the vertex sets of those cliques do not intersect each other.
\end{lem}

Construct on the set $X$ a simple graph $G$ whose edge set consists of all the pairs of points of the space $X$ that are distant from each other by $a$ (``graph of minimal distances''). It is clear that  $1\le k(G)\le\theta(G)\le n-1$.

The following two Lemmas are evident.

\begin{lem}\label{lem:diam}
Let $1<m<n$ and $D=\{X_i\}_{i=1}^m\in\cD_m(X)$. Then $\diam D\in\{a,b\}$, and
\begin{enumerate}
\item $\diam D=a$ if and only if each $X_i$ is a clique in the graph $G$\rom;
\item $\diam D=b$ if and only if one of $X_i$ is not a clique in the graph $G$.
\end{enumerate}
\end{lem}

\begin{lem}\label{lem:alpha}
Let $1<m<n$ and $D=\{X_i\}_{i=1}^m\in\cD_m(X)$. Then $\a(D)\in\{a,b\}$, and
\begin{enumerate}
\item $\a(D)=b$ if and only if each connected component of the graph $G$ is contained in some  $X_i$\rom;
\item $\a(D)=a$ if and only if some connected component of the graph $G$ intersects several subsets $X_i$.
\end{enumerate}
\end{lem}

\begin{cor}\label{cor:kEQtheta}
Let $1<m<n$ and $k:=k(G)=\theta(G)$, then
$$
2d_{GH}(\l\D,X)=
\begin{cases}
\max\{b-\l,\,b,\,\l-b\}&\text{for $m<k$},\\
\max\{b-\l,\,a,\,\l-b\}&\text{for $m=k$},\\
\max\{b-\l,\,a,\,\l-a\}&\text{for $m>k$}.
\end{cases}
$$
\end{cor}

\begin{proof}
Due to Lemma~\ref{lem:ComponEdClique}, the graph $G$ is partitioned into $k$ cliques.

If $m<k$, then each partition $D=\{X_i\}\in\cD_m(X)$ contains an $X_i$ which is not a clique, therefore, due to Lemma~\ref{lem:diam}, $\diam D=b$. On the other hand, there exists a partition $D$ such that each component of the graph $G$ is contained in some $X_i$, and hence, due to Lemma~\ref{lem:alpha}, $\a(D)=b$. Therefore, $(b,b)\in\cAD_m(X)$, and so, due to Lemma~\ref{lem:extremeCases}, we get $\Ext_m(X)=\bigl\{(b,b)\bigr\}$, that implies the first formula.

If $m=k$, then for a partition $D\in\cD_m(X)$ into $m$ cliques we have $\diam D=a$ and $\a(D)=b$, therefore, $(b,a)\in\cAD_m(X)$, and so, due to Lemma~\ref{lem:extremeCases}, we conclude that $\Ext_m(X)=\bigl\{(b,a)\bigr\}$, that implies the second formula.

At last, if $m>k$, then for any partition $D=\{X_i\}\in\cD_m(X)$ some component of the graph $G$ intersects different $X_i$, and so, due to Lemma~\ref{lem:alpha}, the equality $\a(D)=a$ holds for any such partition. On the other hand, there exists a partition $D$ such that each its component is a clique, therefore, for such partition we have $\diam D=a$. So, $(a,a)\in\cAD_m(X)$, and hence, due to Lemma~\ref{lem:extremeCases}, we obtain $\Ext_m(X)=\bigl\{(a,a)\bigr\}$, that implies the third formula.
\end{proof}

It remains to consider the case $k(G)<\theta(G)$.

\begin{cor}
Let $1<m<n$ and $k:=k(G)<\theta(G)=:\theta$, then
$$
2d_{GH}(\l\D,X)=
\begin{cases}
\max\{b-\l,\,b,\,\l-b\}&\text{for $m\le k$},\\
\max\{b-\l,\,b,\,\l-a\}&\text{for $k<m<\theta$},\\
\max\{b-\l,\,a,\,\l-a\}&\text{for $m\ge\theta$}.
\end{cases}
$$
\end{cor}

\begin{proof}
The cases $m<k$ and $m>\theta$ can be proved just as the first and the last formulas from Corollary~\ref{cor:kEQtheta}, respectively.

Let $m=k<\theta$, then as in the proof of Corollary~\ref{cor:kEQtheta} the equality $\diam D=b$ holds for each partition $D\in\cD_m(X)$. Since $m=k$, then the partition $D$ into $X_i=V(G_i)$ belongs to $\cD_m(X)$, and hence $\a(D)=b$. But then $\Ext_m(X)=\bigl\{(b,b)\bigr\}$ in accordance with Lemma~\ref{lem:extremeCases}, that proves the first formula.

Let $k<m<\theta$. In this case we have $\diam D=b$ and $\a(D)=a$ for any $D\in\cD_m(X)$. Thus, in this case $\Ext_m(X)=\cAD_m(X)=\bigl\{(a,b)\bigr\}$, that implies the second formula.

At last, let $m=\theta$, then for any $D\in\cD_m(X)$ there exists a component of the graph $G$ intersecting more than one element of the partition $D$, so, due to Lemma~\ref{lem:alpha}, we have $\a(D)=a$. However, for the partition $D$ of the graph $G$ into $m$ cliques the equality $\diam D=a$ holds, and hence, $(a,a)\in\cAD_m(X)$, and so we have $\Ext_m(X)=\bigl\{(a,a)\bigr\}$ in accordance with Lemma~\ref{lem:extremeCases}, that implies the third formula.
\end{proof}

Collect all the above results.

\begin{thm}\label{thm:main}
Let $X$ be a finite $2$-distance space with non-zero distances $a$ and $b$, $a<b$, $n=\#X$, and let $\l\D$ be a simplex, $m=\#\l\D$. By $G$ we denote the graph with the vertex set $X$ and the edge set consisting of all the pairs of points from $X$ that are distant from each other by $a$. Let $k:=k(G)$ be the number of connected components of the graph $G$, and let $\theta:=\theta(G)$ be its clique covering number. Then
\begin{enumerate}
\item If $k=\theta$, then
$$
2d_{GH}(\l\D,X)=
\begin{cases}
b&\text{for $m=1$},\\
\max\{b,\,\l-b\}&\text{for $1<m<k=\theta$},\\
\max\{b-\l,\,a,\,\l-b\}&\text{for $m=k=\theta$},\\
\max\{b-\l,\,a,\,\l-a\}&\text{for $k=\theta<m<n$},\\
\max\{b-\l,\,\l-a\}&\text{for $m=n$},\\
\max\{b-\l,\,\l\}&\text{for $m>n$};
\end{cases}
$$
\item If $k<\theta$, then
$$
2d_{GH}(\l\D,X)=
\begin{cases}
b&\text{for $m=1$},\\
\max\{b,\,\l-b\}&\text{for $1<m\le k$},\\
\max\{b,\,\l-a\}&\text{for $k<m<\theta$},\\
\max\{b-\l,\,a,\,\l-a\}&\text{for $\theta\le m<n$},\\
\max\{b-\l,\,\l-a\}&\text{for $m=n$},\\
\max\{b-\l,\,\l\}&\text{for $m>n$}.
\end{cases}
$$
\end{enumerate}
\end{thm}

In paper~\cite{IvaTuzBorsuk} we considered \emph{Generalized Borsuk Problem\/}: Is it possible to partition a given bounded metric space $X$ into a given number of subsets having strictly less diameter than $X$ has. The following theorem is proved in~\cite{IvaTuzBorsuk}.

\begin{thm}[\cite{IvaTuzBorsuk}]\label{thm:Borsuk}
Let $X$ be an arbitrary bounded metric space, and $m$ a cardinal number with $m\le\#X$. Choose an arbitrary number $l$, $0<\l<\diam X$, then
\begin{enumerate}
\item The space $X$ can be partitioned into $m$ subsets of strictly less diameters if and only if  $2d_{GH}(\l\D_m,X)<\diam X$\rom;
\item The space $X$ can not be partitioned into $m$ subsets of strictly less diameters if and only if  $2d_{GH}(\l\D_m,X)=\diam X$.
\end{enumerate}
\end{thm}

Theorems~\ref{thm:main} and~\ref{thm:Borsuk} imply the following solution to Generalized Borsuk Problem for finite $2$-distance spaces.

\begin{cor}\label{cor:Borsuk}
Let $X$ be a finite $2$-distance space with non-zero distances $a$ and $b$, $a<b$, and $n=\#X$. By $G$ we denote the graph with vertex set $X$ and edge set consisting of all pairs of points from $X$ distant from each other by $a$. Let $\theta(G)$ be the clique covering number of the graph $G$. Then the space $X$ can not be partitioned into $m$ subsets of strictly less diameters if $m<\theta(G)$, and can be partitioned into $m$ subsets of strictly less diameters if  $\theta(G)\le m\le n$.
\end{cor}

Reformulate Corollary~\ref{cor:Borsuk} and give the exact value of the clique covering number in terms of the Gromov--Hausdorff distance.

\begin{cor}\label{cor:clique}
Let $G=(V,E)$ be an arbitrary finite graph. Define a metric on $V$ as follows\/\rom: the distance between two vertices of $G$ is equal to $a$ if and only if they are adjacent in $G$, otherwise it equals $b$, with $a<b\le2a$. Let $m$ be the greatest positive integer such that $2d_{GH}(a\D,V)=b$, where $a\D$ is a simplex consisting of $m$ vertices. Then $\theta(G)=m+1$.
\end{cor}

Let us give another Corollary that permits to calculate the chromatic number of a graph in terms of the Gromov--Hausdorff distance.

\begin{cor}\label{cor:chrom}
Let $G=(V,E)$ be an arbitrary finite graph. Define a metric on $V$ as follows\/\rom: the distance between two vertices of $G$ is equal to $b$ if and only if they are adjacent in $G$, otherwise it equals $a$, with $a<b\le2a$. Let $m$ be the greatest positive integer such that $2d_{GH}(a\D,V)=b$, where $a\D$ is a simplex consisting of $m$ vertices. Then $\g(G)=m+1$.
\end{cor}

\end{document}